\documentclass{article}[12]
\usepackage[utf8]{inputenc}
\usepackage[english]{babel}
\usepackage{mathtools}
\usepackage{amsthm}
\usepackage{authblk}
\usepackage[margin=1in]{geometry}

\newtheorem{Theorem}{Theorem}[section]

\newtheorem{Lemma}{Lemma}[section]

\begin{document}

\title{Four Blocks Cycles $C(k,1,1,1)$ in Digraphs}
\author{Zahraa Mohsen $^{1,2}$}
	
	\footnotetext[1]{Lebanese University, KALMA Laboratory, Baalbeck.}	
	\footnotetext[2]{University of Paris, IMJ Laboratory, Paris.}
	\date{}
\maketitle

  \begin{center}
\begin{abstract}
 A four blocks cycle $C(k_{1},k_{2},k_{3},k_{4})$ is an oriented cycle formed by the union of four internally disjoint directed paths of lengths $k_{1},k_{2},k_{3},$ and $k_{4}$ respectively. El Mniny\cite{Darine} proved that if D is a digraph having a spanning out-tree T with no subdivisions of C(k, 1, 1, 1), then the chromatic number of D is at most $8^{3}k$. In this paper, we will improve this bound to $18k$.

\end{abstract}

\end{center}
  
  \section{Introduction}
Throughout this paper, all the graphs are considered to be simple, that is with no loops and no multiple edges. By giving an orientation to each edge of a graph $G$, we obtain an oriented graph called a digraph, denoted by $D$. Reciprocally, the graph obtained from a digraph $D$ by ignoring the directions of its arcs is called the underlying graph of $D$, and denoted by $G(D)$. The chromatic number of a digraph $D$, denoted by $\chi(D)$, is the chromatic number of its underlying graph. In \cite{Cohen}, Cohen et al. proved that for any two digraphs $D_{1}$ and $D_{2}$, we have $\chi(D_{1} \cup D_{2})\leq \chi(D_{1}) \times \chi(D_{2})$.\\
\- \- Let $D$ be a digraph. An out-tree $T$ of $D$ is a sub-digraph of $D$ whose vertices are of in-degree 1 except for one vertex of in-degree 0, which is called the root of $T$.\\
\- \- Let $T$ be a spanning out-tree of a digraph $D$ rooted at $r$. For a vertex $x$ of $D$, there is a unique $rx$-directed path in $T$, denoted by $T_{[r,x]}$. The level of $x$ with respect to $T$, denoted by $l_{T}(x)$, is the length of this path. For a non-negative integer $i$, denote $L_{i}(T)$ to be the set of all the vertices having a level $i$ in $T$.\\
\- \- For $x\in V(T)$, the ancestors of $x$ are the vertices that belong to $T_{[r,x]}$. For an ancestor $y$ of $x$, we will write $y \leq_{T} x$ and we will denote by $T_{[y,x]}$ the directed path in $T$ from $y$ to $x$. For two vertices $x$ and $y$ of $T$, the least common ancestor, $z$, of $x$  and $y$ is the common ancestor of $x$ and $y$ having the highest level in $T$. Let $D'$ be a sub-digraph of a digraph $D$ and $T$ be a spanning out-tree of $D$. Let $x\in V(D')$, $x$ is said to be a minimal in $D'$ for $\leq_{T}$ if $\forall v \in V(D')$ satisfying $v\leq_{T} x$, we have $x \leq _{T} v$. Moreover, $x$ is said to be a maximal in $D'$ for $\leq_{T}$ if $\forall v \in V(D')$ satisfying $x\leq_{T} v$, we have $v \leq _{T} x$.  \\
\- \- An arc $(x,y)$ of $D$ is said to be forward (resp. backward) with respect to $T$ if $l_{T}(x) < l_{T}(y)$ (resp. $l_{T}(x)\geq l_{T}(y)$). For two adjacent vertices $x$ and $y$, we denote by $xy$ the arc between $x$ and $y$ if its orientation is unknown.\\
\- \- A maximal out-tree $T$ of $D$ is a spanning out-tree for which for any backward arc with respect to $T$, say $(x,y)$, there exists a $yx$-directed path in $T$. We can easily see that for a maximal out-tree $T$ of a digraph $D$, $L_{i}(T)$ is stable in $D$ for all $i \geq 0$.\\
\- \-  A digraph $D$ is said to be strong if for any two vertices $u$ and $v$ of $D$, there is a directed path from $u$ to $v$. One can easily notice that every strong digraph $D$ has a spanning out-tree. In deed, suppose not and let $T$ be an out-tree of $D$ with maximal number of vertices. Then there exists a vertex $x \in V(D)-V(T)$. Since $D$ is strong then there exists a vertex $y\in V(T)$ such that $(y,x) \in A(D)$. Let $T'=T\cup (y,x)$, an out-tree of $D$ with $V(T')>V(T)$, a contradiction.\\
\- In \cite{Darine}, El Mniny proved that any digraph $D$ that has a spanning out-tree admits a maximal out-tree. Consequently, every strong digraph admits a maximal out-tree.\\
\- \- Given an oriented path $P$ (resp. otiented cycle $C$), a block is a maximal directed subpath of $P$ (resp. $C$). We denote by $P(k_{1}, k_{2},\cdots, k_{n})$ (resp. $C(k_{1}, k_{2},\cdots,k_{n})$) the oriented path (resp. oriented cycle) formed of n blocks of lengths $k_{1}, k_{2},\cdots,$ and $k_{n}$ respectively. Moreover, for a block $B$ of an oriented cycle $C$ which is directed from $x$ to $y$, we denote by $x $ (resp. $y$) the source (resp. the sink) of $B$, and we write $B$ as $C_{[x,y]}$. Moreover, $x$ and $y$ are called the ends of $B$. In general, for an oriented cycle $C$ with s blocks, we will denote by $B_{i}$'s the blocks of $C$, such that $B_{2i}$ has ends $x_{i+1}$ and $y_{i}$ for all $1\leq i\leq \frac{s}{2}-1$ and $B_{2i+1}$ has ends $x_{i+1}$ and $y_{i+1}$ for all $0\leq i\leq \frac{s}{2}-1$.\\
\- \- A subdivision of a digraph $F$ is a digraph $F^{'}$ obtained by replacing each arc $(x,y)$ of $F$ by an $xy$-directed path of length at least 1.\\
\- \- Let $k_{1},k_{2},k_{3},$ and $k_{4}$ be positive integers. Cohen et al.\cite{Cohen} proved that any strong digraph with no subdivisions of $C(1,1,1,1)$ has chromatic number less than 24 . El Mniny\cite{Darine} proved that for any $k$ positive integer, any digraph with a spanning out-tree with no subdivision of $C(k,1,1,1)$ has a chromatic number at most $8^{3}k$. In this paper, we are going to improve the bound established in \cite{Darine} by proving that if D is a digraph having a spanning out-tree T with no subdivisions of C(k, 1, 1, 1), then the chromatic number of D is at most $18k$.\\
\- \- In our proof, we used for the first time in such investigations what we call a wheel in order to study the chromatic number of the digraphs. A wheel is a graph made up of a chordless cycle and a vertex adjacent to at least three vertices of the cycle, such vertex is called universal. In \cite{Thomassen}, Thomassen et al. proved that every graph with no wheel as a subgraph is 3-colorable. This result allowed us to get lower bounds for the chromatic number of the studied digraph, $D$, by proving that some subdigraphs of $D$ contain no wheels.
 
 \section{Main Result}
\- \- This section is devoted to prove our main result dealing with the existence of four blocks cycles in digraphs having a spanning out-tree.\\
\- \-Let $k$ be a positive integer and let $D$ be a digraph with a spanning out-tree. $D$ admits a maximal out-tree, say $T$.\\
\- \-We used the same partition of $D$ introduced by El Mniny\cite{Darine}. However we found less bounds for the chromatic number of each part, sometimes by using cycles and sometimes by using the wheels that facilitated our study of such digraph.\\
\- \-For $i=0,..,k-1$, let $V_{i}:= \cup_{\alpha\geq 0}L_{i+\alpha k}(T)$. Define $D_{i}$ to be the subdigraph of $D$ induced by $V_{i}$, and then partition the arcs of $D_{i}$ as follows:\\
\- \- \- \- $A_{1}:= \{(x, y)| x \leq_{T} y\}$;\\
\- \- \- \- \- $A_{2}:= \{(x, y)| y \leq_{T} x\}$;\\
\- \- \- \- \- $A_{3}:= A(D_{i}) \backslash (A_{1} \cup A_{2})$.\\
\- \-For $0 \leq i \leq k-1$ and $j = 1, 2, 3$, let $D_{i}^{j}$ be the spanning subdigraph of $D_{i}$ whose arc-set is $A_{j}$.\\

Let  $C$ be a cycle of $D$ of $s$ blocks, say $B_{1},\cdots,B_{s}$ with $s\geq 4$. We will use the same notations for the ends of the blocks of $C$ as introduced in the introduction such that the $x_{i}$'s (resp. $y_{i}$'s) are the sources (resp. sinks) of the blocks of $C$. Without loss of generality, we will suppose that $x_{1}$ is minimal in $\{x_{i}\}_{1\leq i\leq s/2}$ for $\leq_{T}$. Let $0\leq i \leq k-1$, $C$ is said to be a mixed cycle if all its blocks are induced by arcs in $D_{i}^{1}$ except for the block $B_{1}$ that contains a vertex $z_{1} \neq y_{1}$ such that the arcs of $C_{[x_{1},z_{1}]}$ belong to $T$ and the arcs of $C_{[z_{1},y_{1}]}$ belong to $D_{i}^{1}$. $B_{1}$ is called the mixed block of $C$.\\
Remark that $x_{1}$ is the smallest for $\leq_{T}$ in $C$.
\begin{Lemma}
If $D$ contains a mixed cycle then it contains a subdivision of $C(k,1,1,1)$.
  \end{Lemma}          
  \begin{proof} Let $C$ be a mixed cycle of $D$. We will proceed by induction on the number of blocks of $C$, $s$.\\ 
  \-  For $s=4$, if $l(C[x_{1},z_{1}])\geq k$ then $C$ is a subdivision of $C(k,1,1,1)$. Else let $z$ be the minimal for $\leq_{T}$ in $C-C[x_{1},z_{1}]$. We will study three cases depending on the position of $z$ on $C$. In deed, if $z\in C_{]x_{1},y_{1}[}$ then we replace in $C$, $C_{]z_{1},z[}$ by $T_{]z_{1},z[}$ and get a subdivision of $C(k,1,1,1)$. Else if $z\in C_{]x_{1},y_{2}[}$ then $C_{[z_{1},y_{1}]}$ $\cup$ $C_{[x_{2},y_{1}]}$ $\cup$ $C_{[x_{2},y_{2}]}$ $\cup$ $T_{[z_{1},z]}$ $\cup$ $ C_{[z,y_{2}]}$ is a subdivision of $C(k,1,1,1)$. Finally if $z=x_{2}$, in this case we will introduce $z'$ to be a minimal for $\leq_{T}$ in $C-C_{[x_{1},z_{1}]}-\{x_{2}\}$. If $z'$ $\in$ $C_{[x_{2},y_{1}]}$ or $C_{[x_{2},y_{2}]}$ then replace in $C$, $C_{[x_{2},z']}$ by $T_{[x_{2},z']}$ and get a subdivision of $C(k,1,1,1)$. Else if $z'$ $\in$ $C_{[x_{1},y_{2}[}$ then replace in $C$, $C_{]x_{2},y_{2}[}$ $\cup$ $C_{]z',y_{2}[}$ by  $T_{[x_{2},z']}$ and get a subdivision of $C(k,1,1,1)$. The only case left to study is if $z'$ $\in$ $C_{]z_{1},y_{1}[}$, then replace in $C$, $C_{]x_{2},y_{1}[}$ $\cup$ $C_{]z',y_{1}[}$ by $T_{[x_{2},z']}$ and get a subdivision of $C(k,1,1,1)$.\\ 
  \- \- Suppose it is true up to $s-2$ and let's prove it for $s$, $s\geq6$. Notice that $z_{1} \leq_{T} x_{i}$ for all $1\leq i \leq \frac{s}{2}$. Let $i\neq1$ be the integer such that $l_{T}(x_{i})$ is minimal. If $i\geq3$, then $C_{[z_{1},y_{1}]}$ $\cup$ $C_{[x_{2},y_{1}]}$ $\cup$.... $\cup$ $C_{[x_{i},y_{i-1}]}$ $\cup T_{[z_{1},x_{i}]}$ contains a mixed cycle of $D$ of blocks less than or equal to $s-2$. Otherwise $i=2$, we can get a mixed cycle of $D$ of blocks less than or equal to $s-2$ by replacing in $C$, $C_{]z_{1},y_{1}]}$ $\cup$ $C_{]x_{2},y_{1}]}$ by $T_{[z_{1},x_{2}]}$. In both cases, using the induction hypothesis, we get that $D$ contains a subdivision of $C(k,1,1,1)$.
    \end{proof}        
  Consequently, if $D$ contains a cycle $C$ of $s$ blocks, $s\geq4$, whose all arcs are in $D_{i}^{1}$, $0\leq i \leq k-1$, then $D$ contains a subdivision of $C(k,1,1,1)$.\\
  
  Let $C$ be a cycle of $D$ of $s$ blocks, say $B_{1},\cdots,B_{s}$. We will use the same notations for the ends of the blocks of $C$ as introduced in the introduction such that the $y_{i}$'s (resp. $x_{i}$'s) are the sources (resp. sinks) of the blocks of $C$. Let $0\leq i \leq k-1$, $C$ is said to be back-mixed if $s\geq 6$ and all its blocks are induced by arcs in $D_{i}^{2}$ except for $B_{s}$ that contains a vertex $z_{1}$ such that the arcs of $C_{[z_{1},x_{1}]}$ belong to $T$, the arcs of $C_{[y_{\frac{s}{2}},z_{1}]}$ belong to $D_{i}^{2}$, $z_{1}\neq y_{\frac{s}{2}}$, and $z_{1}$ is a minimal in C for $\leq_{T}$. $B_{s}$ is called the back-mixed block of $C$.\\
   Remark that $x_{1}$ is the smallest for $\leq_{T}$ in $C-C_{[z_{1},x_{1}]}$.
           
  \begin{Lemma}
        If $D$ contains a back-mixed cycle then it contains a subdivision of $C(k,1,1,1)$.
   \end{Lemma}  
          
    \begin{proof}Let $C$ be a back-mixed cycle of $D$. We will proceed by induction on the number of blocks of $C$, $s$.
    Notice that $x_{1}\leq_{T}x_{i}$ for all $1\leq i \leq \frac{s}{2}$.\\
   \- \- \-For $s=6$.  Let $i\neq1$ be the integer such that $l_{T}(x_{i})$ is minimal. If $i=2$, then $C_{[y_{2},x_{2}]}$ $\cup$ $C_{[y_{2},x_{3}]}$ $\cup$  $C_{[y_{3},x_{3}]}$ $\cup$ $C_{[y_{3},x_{1}]}$ $\cup$ $T_{[x_{1},x_{2}]}$ contains a subdivision of $C(k,1,1,1)$. Else $i=3$, then $C_{[y_{1},x_{1}]}$ $\cup$ $C_{[y_{1},x_{2}]}$ $\cup$ $C_{[y_{2},x_{2}]}$ $\cup$ $C_{[y_{2},x_{3}]}$ $\cup T_{[x_{1},x_{3}]}$ contains a subdivision of $C(k,1,1,1)$.\\
   \- \- \- Suppose it is true up to $s-2$, and let's prove it for $s$, $s\geq8$. Let $i\neq 1$ be the integer such that $l_{T}(x_{i})$ is minimal. If $i>3$, then $C_{[y_{1},x_{1}]}$ $\cup$ $C_{[y_{1},x_{2}]}$ $\cup\cdots \cup$ $C_{[y_{i-1},x_{i}]}$ $\cup$ $T_{[x_{1},x_{i}]}$ contains a back-mixed cycle of $D$ with blocks less than or equal to $s-2$, and so $D$ contains a subdivision of $C(k,1,1,1)$. Else if $i=2$, then replace in $C$, $C_{]y_{1},x_{1}[}$ $\cup$ $C_{]y_{1},x_{2}[}$ by $T_{[x_{1},x_{2}]}$, this contains a back-mixed cycle of $D$ of blocks less than or equal $s-2$, and so $D$ contains a subdivision of $C(k,1,1,1)$. Finally if $i=3$, notice that $C_{[y_{1},x_{1}]}$ $\cup$ $C_{[y_{1},x_{2}]}$ $\cup$ $ C_{[y_{2},x_{2}]}\; \cup$ $C_{[y_{2},x_{3}]}$ $\cup$ $T_{[x_{1},x_{3}]}$ contains a subdivision of $C(k,1,1,1)$.
 \end{proof}
  Consequently, if $D$ contains a cycle $C$ of $s$ blocks, $s\geq6$, whose all arcs are in $D_{i}^{2}$, $0\leq i \leq k-1$, then $D$ contains a subdivision of $C(k,1,1,1)$.\\
  
  Let $C$ be a cycle of $D$ of $4$ blocks, say $B_{1},\cdots,B_{4}$. We will use the same notations for the ends of the blocks of $C$ as introduced in the introduction such that the $y_{i}$'s (resp. $x_{i}$'s) are the sources (resp. sinks) of the blocks of $C$. $C$ is said to be a bad 4-blocks cycle if it is either a $C(1,1,1,1)$ ( named bad of type 1) or a $C(2,1,1,1)$ such that $B_{1}=(y_{1},z_{1}) \cup (z_{1},x_{1})$ and $x_{1}\leq_{T} x_{2}\leq_{T} z_{1}\leq_{T}y_{2}\leq_{T} y_{1}$ ( named bad of type 2). Else $C$ is said to be a good 4-blocks cycle.
  \begin{Lemma} Let $C$ be a cycle with 4-blocks in $D_{i}^{2}$, $0\leq i \leq k-1$. If $D$ contains no subdivision of $C(k,1,1,1)$, then $C$ is a bad cycle.  \end{Lemma}
   \begin{proof}
    Suppose to the contrary that $C$ is a good 4-blocks cycle.
    In deed, consider without loss of generality $x_{1}$ to be a minimal in $C$ for $\leq_{T}$. Let $z$ minimal in $C$-$\{x_{1}\}$ for $\leq_{T}$. If $z\in C_{]y_{i},x_{1}[}$ for $i=1 $ or 2 then replace in $C$, $(z,x_{1})$ by $T_{[x_{1},z]}$ and get a subdivision of $C(k,1,1,1)$, a contradiction. Then $z=x_{2}$. Let $z'$ be a minimal in $C-\{x_{1},x_{2}\}$ for $\leq_{T}$. If $z'\in C_{]y_{i},x_{2}[}$ for $i =1 $ or $2$ then replace $(z',x_{2})$ by $T_{[x_{2},z']}$ and get a subdivision of $C(k,1,1,1)$, a contradiction. Notice also that $y_{1}$ and $y_{2}$ can not be both minimal in $C-\{x_{1},x_{2}\}$ for $\leq_{T}$, since else $C$ is bad of type 1, a contradiction. Moreover if without loss of generality $y_{1}$ is minimal in $C$-$\{x_{1},x_{2}\}$ for $\leq_{T}$, then let $z''$ be a maximal in $C-\{y_{2}\}$ for $\leq_{T}$ such that $z''\leq_{T} y_{2}$. If $z''\in C_{]y_{2},x_{i}[}$ for $i=1$ or $2$ then replace in C, $(y_{2},z'')$  by $T_{[z'',y_{2}]}$ and get a subdivision of $C(k,1,1,1)$, a contradiction. Then $z''=y_{1}$ or $z''=x_{2}$ and so $C$ is bad of type 1, a contradiction. Hence $z' \notin C_{[y_{i},x_{2}[}$ for $i =1,2$. If without loss of generality $z'\in C_{]y_{1},x_{1}[}$, let $z''$ be a minimal in $C-\{x_{1},x_{2},z'\} $ for $\leq_{T}$. If $z''=y_{1}$ then replace in $C$, $(y_{1},z')$ by $T_{[z',y_{1}]}$ and get a subdivision of $C(k,1,1,1)$, a contradiction. Else if $z''\in C_{]y_{i},x_{2}[}$ for $i=1,2$ then $C_{[z',x_{1}]} \cup T_{[z',z'']} \cup C_{[y_{2},z'']} \cup C_{[y_{2},x_{1}]} $ is a subdivision of $C(k,1,1,1)$, a contradiction.
    Else if $z''\in C_{]y_{2},x_{1}[}$ then replace in $C$, $C_{[z',x_{1}]} \cup C_{[z'', x_{1}]} $ by $T_{[z',z'']}$ and get a subdivision of $C(k,1,1,1)$, a contradiction. Else if $z''\in C_{]y_{1},z'[}$ then $C_{[z',x_{1}]} \cup T_{[z',z'']} \cup C_{[y_{1},z'']} \cup C_{[y_{1},x_{2}]} \cup T_{[x_{1},x_{2}]}$ is a subdivision of $C(k,1,1,1)$, a contradiction. Then $z''=y_{2}$ then let $z'''$ be a maximal in $C$-$\{y_{1}\}$ with $z'''\leq_{T} y_{1}$. Notice that $z'''$ can not be $y_{2}$ since else $C$ is bad of type 2, where $x_{1}\leq_{T} x_{2}\leq_{T} z' \leq_{T}y_{2}\leq_{T} y_{1}$, a contradiction. Hence $z'''\in C_{]y_{1},x_{2}[}$ or $C_{]y_{1},z'[}$ and so replace in $C$, $(y_{1},z''')$ by $T_{[z''',y_{1}]}$ and get a subdivision of $C(k,1,1,1)$, a contradiction. \\
  \end{proof}
  \begin{Lemma}
   If $D$ has no subdivision of $C(k,1,1,1)$ then $\chi(D_{i}^{3}) \leq 2$, for all $0\leq i \leq k-1$.
\end{Lemma}
  \begin{proof}
   We claim that the underlying graph of $D_{i}^{3}$ is bipartite. In deed, suppose to the contrary that $D_{i}^{3}$ contains an odd cycle $C=x_{1}...x_{t}$. Without loss of generality, suppose that $x_{1}$ is with minimal level for $T$ in $C$. Now we will study two cases:\\
  \-\- \textbf{Case 1:}
  Neither $x_{t}$ is ancestor of $x_{2}$ nor $x_{2}$ is ancestor of $x_{t}$.\\
  \- \- In this case $l_{T}(x_{t})=l_{T}(x_{2})$, since else let $y$ be their least common ancestor, so $T_{[y,x_{2}]} \cup T_{[y,x_{t}]} \cup (x_{1},x_{t}) \cup (x_{1},x_{2})$ is a subdivision of $C(k,1,1,1)$, a contradiction. Thus $t\geq5$. Similarly, we can see that both $T_{[y,x_{t}]}$ and $T_{[y,x_{2}]}$ have lengths less than $k$.\\
  \- \- Notice that $x_{2}$ is an ancestor of $x_{t-1}$. If not, let $z$ be the least common ancestor of $x_{2}$ and $x_{t-1}$. If $x_{1} \notin T_{[z,x_{t-1}]}$ then $T_{[z,x_{t-1}]}$ $\cup$ $T_{[z,x_{2}]}$ $\cup \;(x_{1},x_{2})$ $\cup \;(x_{1},x_{t})$ $\cup \; x_{t-1}x_{t}$ is a subdivision of $C(k,1,1,1)$, a contradiction. Else, $T_{[y,x_{t}]}$ $\cup$ $x_{t-1}x_{t}$ $\cup$ $T_{[y,x_{2}]}$ $\cup \;(x_{1},x_{2})$ $\cup$ $T_{[x_{1},x_{t-1}]}$ is subdivision of $C(k,1,1,1)$, a contradiction.\\
  \- \- As well, $x_{t}$ is an ancestor of $x_{3}$. Thus, $(x_{2},x_{3})$ and $(x_{t},x_{t-1})$ are arcs in $D_{i}^{3}$. Then $T_{[x_{t},x_{3}]} \cup T_{[x_{2},x_{t-1}]} \cup (x_{2},x_{3}) \cup (x_{t},x_{t-1})$ is a subdivision of $C(k,1,1,1)$, a contradiction.\\
  \-\- \textbf{Case 2:} Without loss of generality, suppose that $x_{2} \leq_{T} x_{t}$.\\
  Let $i$ the smallest integer greater than $2$ satisfying $l_{T}(x_{i}) > l_{T}(x_{i-1})$. We will study the following cases:\\
       \- \- \- If $i=3$. Consider the least common ancestor of $x_{1}$ and $x_{3}$, y, then $T_{[y,x_{1}]}$ $\cup(x_{1},x_{t})$  $\cup$ $T_{[y,x_{3}]}$ $\cup(x_{2},x_{3})  \cup T_{[x_{2},x_{t}]}$ is a subdivision of $C(k,1,1,1)$, a contradiction. Notice that $D_{i}^{3}$ has no path of type $P(1,2)$ satisfying the same properties of the $x_{t}x_{1}x_{2}x_{3}$ which is a path of type $P(1,2)$ with $x_{2} \leq_{T} x_{t}$, since else we can find similarly a subdivision of $C(k,1,1,1)$ in $D$, a contradiction.\\
       \- \- \- If $i=t$, let $z$ the least common ancestor of $x_{1}$ and $x_{t-2}$, then $T_{[z,x_{1}]} \cup (x_{1},x_{t}) \cup T_{[z,x_{t-2}]} \cup (x_{t-1},x_{t-2}) \cup (x_{t-1},x_{t})$ is a subdivision of $C(k,1,1,1)$, a contradiction. \\
       \- \- \-If $4\leq i<t$. Then $l_{T}(x_{i}) > l_{T}(x_{i-2})$, since else we can consider the least common ancestor of $x_{1}$ and $x_{i}$ to find a subdivision of $C(k,1,1,1)$, a contradiction. Hence  $l_{T}(x_{i}) > l_{T}(x_{i-2})$. Notice that $x_{i-2} \leq_{T} x_{i}$, since else let $y$ be their least common ancestor, and so $T_{[y,x_{i-2}]} \cup T_{[y,x_{i}]} \cup (x_{i-1},x_{i}) \cup (x_{i},x_{i-2})$ is a subdivision of $C(k,1,1,1)$, a contradiction. Note that here $i=4$, since else one can notice that $(x_{i-1},x_{i}) \cup (x_{i-1},x_{i-2}) \cup (x_{i-2},x_{i-3})$ is a path $P(1,2)$ as in case $i=3$, a contradiction. Notice that neither $x_{4}\leq_{T}x_{t}$ nor $x_{t}\leq_{T}x_{4}$, since else we can combine the directed path in $T$ between $x_{4}$ and $x_{t}$ and the oriented path in $C$ between $x_{4}$ and $x_{t}$ and find a subdivision of $C(k,1,1,1)$, a contradiction. Moreover, using the least common ancestor of $x_{1}$ and $x_{3}$, we can prove that $l_{T}(x_{4})=l_{T}(x_{t})$ and the length of the directed paths in $T$ from their least common ancestor to each is less than $k$. Now, using the least common ancestor of $x_{4}$ and $x_{t}$ we can show similarly that $l_{T}(x_{1})=l_{T}(x_{3})$. Notice that $t \geq 7$.\\
       Denote by $y$ the least common ancestor of $x_{5}$ and $x_{t}$. Notice here that neither $x_{1}$ nor $x_{3}$ is ancestor of $x_{t}$.\\
       Also $y\neq x_{t}$, since else we can find a subdivision of $C(k,1,1,1)$, a contradiction. Besides, $y\neq x_{5}$, since else we can use the least common ancestor of $x_{1}$ and $x_{3}$ to find a subdivision of $C(k,1,1,1)$, a contradiction.\\
        We claim also that neither $x_{1}$ nor $x_{3}$ is ancestor of $x_{5}$. In deed, if $x_{1} \leq_{T} x_{5}$, we'll consider the least common ancestor of $x_{4}$ and $x_{t}$ to find a subdivision of $C(k,1,1,1)$, a contradiction. Similarly, if $x_{3}\leq_{T} x_{5}$, we'll consider the least common ancestor of $x_{4}$ and $x_{t}$ and the least common ancestor of $x_{1}$ and $x_{3}$ to find a subdivision of $C(k,1,1,1)$, a contradiction. \\
      Hence denote by $z$ the least common ancestor of $x_{1}$ and $x_{3}$. If $y\notin T_{[z,x_{i}]}$, $i=1,3$, then $T_{[z,x_{3}]}\; \cup \; (x_{3},x_{4})\; \cup \; x_{4}x_{5} \; \cup \;  T_{[y,x_{5}]}\; \cup\; T_{[y,x_{t}]}\; \cup \; T_{[z,x_{1}]} \;\cup \; (x_{1},x_{t})$ is a subdivision of $C(k,1,1,1)$, a contradiction. Else, $y$ is in particular the least common ancestor of $x_{2}$ and $x_{5}$ and $l(T_{[y,x_{2}]}) \geq k$, then $ T_{[y,x_{2}]}\; \cup\; T_{[y,x_{5}]} \cup x_{5}x_{4} \cup (x_{3},x_{4}) \cup (x_{3},x_{2})$ is a subdivision of $C(k,1,1,1)$, a contradiction.\\
\- \- \- Hence $D_{i}^{3}$ contains no odd cycle  then $\chi(D_{i}^{3}) \leq2$.\\
 \end{proof}
 
Now we are ready to prove our main result:   
\begin{Theorem}
 Let $k$ be a positive integer and $D$ a digraph with a spanning out-tree with no subdivisions of $C(k,1,1,1)$ then the chromatic number of $D$ is at most $18k$.
 \end{Theorem}
 
   \begin{proof} Let $D$ be a digraph with a spanning out-tree with no subdivisions of $C(k,1,1,1)$. Let $T$ be a maximal out-tree of $D$. We will consider the same partition for the arcs of $D$ used in the beginning of this section.
 \textbf{Claim 1:} $\chi(D_{i}^{1}) \leq3$, for all $0\leq i \leq k-1$.
 \begin{proof} Suppose to the contrary that $\chi(D_{i}^{1}) >3$, then $D_{i}^{1}$ contains a wheel of cycle say $C$ and a universal vertex $x$. $C$ can't be a cycle of 4 blocks or more, since else $D$ contains a subdivision of $C(k,1,1,1)$, a contradiction. 
 Then $C$ is a cycle of 2 blocks. Denote by $z_{1}$(resp. $z_{2}$) the vertex of $C$ with in-degree (resp. out-degree) zero. We will reach a contradiction finding a subdivision of $C(k,1,1,1)$ by studying the order of the levels of $x$ and its neighbors. Denote by $x_{1}$,$x_{2}$,$x_{3}$ three neighbors of $x$ on $C$ such that $l_{T}(x_{1})< l_{T}(x_{2})< l_{T}(x_{3})$.\\
  If $l_{T}(x)< l_{T}(x_{2})$, then since $z_{1}$ $\notin$ $\{x_{2},x_{3}\}$, $C_{[z_{1},x_{2}]}$ $\cup (x,x_{2}) \cup (x,x_{3}) \cup C[z_{1},x_{3}]$  is cycle with 4 blocks in $D_{i}^{1}$ and so $D$ contains a subdivision of $C(k,1,1,1)$, a contradiction. Then $l_{T}(x_{2})< l_{T}(x)$. If $x_{1}$ and $x_{2}$ belong to the same block on $C$, then $(C-C_{[x_{1},x_{2}]}) \cup (x_{1},x) \cup (x_{2},x)$ is a cycle with 4 blocks in $D_{i}^{1}$ and so $D$ contains a subdivision of $C(k,1,1,1)$, a contradiction. Else $C_{[x_{1},z_{2}]}$ $\cup$ $C_{[x_{2},z_{2}]}$ $\cup$ $(x_{2},x)$ $\cup$ $(x_{1},x)$ is a cycle with 4 blocks in $D_{i}^{1}$ and so $D$ contains a subdivision of $C(k,1,1,1)$, a contradiction.  \end{proof}
  
  \hspace{-7mm}\textbf{ Claim 2:} $\chi(D_{i}^{2}) \leq3$, for all $0\leq i \leq k-1$.
   \begin{proof}      
  Suppose to the contrary that $\chi(D_{i}^{2}) >3$, then $D_{i}^{2}$ contains a wheel of cycle say $C$ and a universal vertex $x$. Notice that $D_{i}^{2}$ contains no good 4-blocks cycle or a cycle with 6 blocks or more, since else $D$ contains a subdivision of $C(k,1,1,1)$, a contradiction. In particular, $C$ can't be neither a good 4-blocks cycle nor a cycle with 6 blocks or more.\\
    \- \- If $C$ is a cycle with 2 blocks, then denote by $z_{1}$ (resp. $z_{2}$) the vertex of $C$ with out-degree (resp. in-degree) zero. We will reach a contradiction finding a subdivision of $C(k,1,1,1)$ by studying the order of the levels of $x$ and its neighbors. Let $x_{1}$,$x_{2}$,$x_{3}$ to be three neighbors of $x$ on $C$ /: $l_{T}(x_{1})< l_{T}(x_{2})< l_{T}(x_{3})$. If $l_{T}(x)< l_{T}(x_{1})$, then $T_{[x_{1},x_{2}]}$ $\cup C_{[x_{2},x_{3}]}\cup (x_{1},x) \cup (x_{3},x)$ is a subdivision of $C(k,1,1,1)$, a contradiction. Else if $l_{T}(x)> l_{T}(x_{3})$. If $x_{1}$ and $x_{2}$ belong to the same block on $C$, then $T[x_{2},x_{3}]$ $\cup$ $C_{[x_{2},x_{1}]}$ $\cup (x,x_{1}) \cup (x,x_{3})$ is a subdivision of $C(k,1,1,1)$. Else $T[x_{2},x_{3}]$ $\cup$ $C_{[x_{2},z_{1}]}$ $\cup C_{[x_{1},z_{1}]}$ $\cup (x,x_{1}) \cup (x,x_{3})$ is a subdivision of $C(k,1,1,1)$.
    Else if $l_{T}(x_{1})< l_{T}(x)< l_{T}(x_{2})$, then let $z$ to be of minimal level in $T$ /: $z\in$ $C$ and $x\leq_{T} z$. It is clear that $z\neq z_{2}$. So replace by $T_{[x,z]}\cup(x,x_{1})$, $C_{[z,x_{1}]}$ (resp.$C_{[z,z_{1}]}\cup C_{[x_{1},z_{1}]}$) if $x_{1}$ and $z$ belong to the same block on C (resp. if $x_{1}$ and $z$ belong to different blocks on C), and get a subdivision of $C(k,1,1,1)$, a contradiction. Then $l_{T}(x_{2})< l_{T}(x)< l_{T}(x_{3})$, let $z$ to be of minimal level in $T$ /: $z\in$ $C$ and $x\leq_{T} z$. If $z\neq z_{2}$, we will proceed as the case before. Else, $z=z_{2}$ and so $z=z_{2}=x_{3}$. Let $z'$ in $C$ of maximal level in $T$ such that $z'\leq_{T} x$. Then $(x_{3},x) \cup T[z',x] \cup C_{[x_{3},z_{1}]}\cup C_{[z',z_{1}]}$ is a subdivision of $C(k,1,1,1)$, a contradiction.\\
    \- \- Hence $C$ is a bad 4-blocks cycle. We will study the following cases:\\
     \- \-  If $C$=$C(2,1,1,1)$ where $x_{1}\leq_{T} x_{2}\leq_{T} z_{1}\leq_{T}y_{2}\leq_{T} y_{1}$, then since $x$ has at least 3 neighbors in $C$ which is of length 5, then we can see that there are two adjacent vertices of $C$ which are both neighbors of $x$. Notice that whenever two adjacent vertices of $C$, $a$ and $b$, are both neighbors of $x$ then replace in $C$, the arc $(a,b)$ by the arcs $ax$ and $bx$ and so we get in $D_{i}^{2}$ a cycle with 6 blocks or a good 4-blocks cycle, a contradiction.\\
      \- \-  If $C=C(1,1,1,1)$ with $y_{1}$ and $y_{2}$ are not ancestors one of the other, if $x_{1}$ and $y_{1}$ are both neighbors of $x$ then replace in $C$, $(y_{1},x_{1})$ by $\{xx_{1}\} \cup \{xy_{1}\}$, we will get a $C(2,1,1,1)$ of good type in $D_{i}^{2}$, a contradiction. Then $x_{2}$ and $y_{2}$ are both neighbors of $x$. Then replace in $C$, $(y_{2},x_{2})$ by $\{xx_{2}\} \cup \{xy_{2}\}$, and so we will get a $C(2,1,1,1)$ of good type in $D_{i}^{2}$, a contradiction.\\
      \- \- If $C=C(1,1,1,1)$ with $x_{1}\leq_{T} x_{2}\leq_{T} y_{2}\leq_{T} y_{1}$, we claim that $x_{1}$ and $y_{1}$ can't be both neighbors of $x$, since else replace in $C$, $(y_{1},x_{1})$ by $\{xx_{1}\} \cup \{xy_{1}\}$. We will either get a $C(2,1,1,1)$ of good type in $D_{i}^{2}$, a contradiction, or we will get a bad 4-blocks cycle of type 2, $C(2,1,1,1)$, and so in this case we have $T_{[x_{2},x]} \cup (y_{1},x) \cup  (y_{1},x_{1}) \cup (y_{2},x_{1}) \cup (y_{2},x_{2}) $ is a subdivision of $C(k,1,1,1)$ in $D$, a contradiction. Then $x_{2}$ and $y_{2}$ are both neighbors of $x$ and so replace in $C$, $(y_{2},x_{2})$ by $\{xx_{2}\} \cup \{xy_{2}\}$. We will either get a $C(2,1,1,1)$ of good type in $D_{i}^{2}$, a contradiction, or we will get a bad 4-blocks cycle of type 2, $C(2,1,1,1)$, in $D_{i}^{2}$ and in this case replace in $C$, $ (y_{1},x_{2})$ by $T_{[y_{1},x]} \cup (x,x_{2})$ and get a subdivision of $C(k,1,1,1)$ in $D$, a contradiction.\\ 
       \- \- If $C=C(1,1,1,1)$ with $x_{1}\leq_{T} x_{2}\leq_{T} y_{1}\leq_{T} y_{2}$, we proceed similarly as the case before.
      \end{proof}      
   
  With the fact that $D_{i} = D^{1}_{i}\cup D^{2}_{ i}\cup D^{3}_{ i}$ then $\chi(D_{i})\leq 2.3.3=18$ for all $i \in \{0,...,k-1\}$. Consequently, as $V (D_{i})$, $0 \leq i \leq k-1$ form a partition of $V (D)$, we obtain a proper $18k$-coloring of $D$ by giving to each $D_{i}$ 18 distinct colors. This implies the hoped result.\\
  \end{proof}
  \newpage
\textbf{Acknowledgment}The author would like to thank A. El Sahili and M. Mortada for their valuable remarks. The author would like also to acknowledge the National Council for Scientific Research of Lebanon (CNRS-L) and the Agence Universitaire de la Francophonie in cooperation with Lebanese University for granting a doctoral fellowship to Zahraa Mohsen.

\begin{center}

\end{center}
\end{document}